 \documentclass[preprint,12pt]{elsarticle}
\usepackage[centertags]{amsmath}
\usepackage{amsfonts}
\usepackage{amssymb}
\usepackage{amsthm}
\usepackage{newlfont}
\usepackage{graphicx}

\newtheorem{theorem}{Theorem}[section]
\newtheorem{lemma}[theorem]{Lemma}
\newtheorem{corollary}[theorem]{Corollary}

\newtheorem{remmark}[theorem]{Remmark}

\theoremstyle{definition}

\usepackage{amssymb}

 \journal{}

\begin{document}

 \begin{frontmatter}



  \fntext[lable3]{Corresponding
author. \textit{E-mail
addresses: \small{m$_-$nabiei@sbu.ac.ir}}}

 \title{Complex symmetric operators acting between two different Hilbert spaces}

 \author{\small{Mona Nabiei.$^{1}$ 
}}

 \address{$^{1}$Department of Mathematics, Shahid Beheshti University, G. C. P.O. Box 19839, Tehran, IRAN
}

 \begin{abstract}
This paper will initiate a study on the class of complex symmetric operators acting between two different Hilbert space. Among other things, we compute the closure of $CSO_u$ and $CSO_b$ with respect to the several topologies.
 \end{abstract}


 \end{frontmatter}

\section{Introduction }
In linear algebra, a symmetric matrix is a square matrix that is equal to its transpose. But, in operator theory, a symmetric operator on a complex Hilbert space $H$ is an operator that is its own adjoint. So, if $H$ is finite-dimensional with a given orthonormal basis, this is equivalent to the condition that the matrix of $A$ is Hermitian, i.e., equal to its conjugate transpose $A^*$. 

But, the study of complex symmetric operators has been flourished near the intersection of operator theory and complex analysis. The term complex symmetric stems from the fact that $T$ is a complex symmetric operator on a Hilbert space $H$ if and only if there is an orthonormal basis of $H$ with respect to which $T$ has a symmetric matrix representation with complex entries.  

About half a century age, Glazman laid the foundations for the theory of unbounded complex symmetric operators \cite{fofez, sofez}. Since then, his fundamental ideas have been successfully tested on several classes of differential operators \cite{seofez, twofez, twseofez}. The general study of complex symmetric operators was undertaken by S.R. Garcia in \cite{garfdef, gardef, gar5}.

One of his work, that is of interest in this paper, is on the norm closure problem for bounded complex symmetric operators. He proves that the set of all bounded complex symmetric operators on a separable, infinite-dimensional Hilbert space is not norm closed \cite{garclo, ez}. Recently many other researchers have obtained some results for the set of all complex symmetric operators. For example, it is shown that this set is transitive and 2-hyperreflexive with constant 1 \cite{gp}.

In this article, we look at this from a different point of view. At first, we define a suitable metric $d$ on $C(H,K)$, the set of all closed densely-defined linear operators from Hilbert space $H$ to finite dimensional Hilbert space $K$. Already, metric geometric properties, such as compactness of the class of admissible sets, metric convexity, and normal structure are examined \cite{mon}. Then, we introduce the notion of complex symmetric operator (not necessarily bounded) acting between two different Hilbert spaces and study the closure $\overline{CSO}$ of the set $CSO$ of all densely-defined complex symmetric operators from Hilbert space $H$ to finite dimensional Hilbert space $K$, with respect to this new effective metric.

\section{Preliminary Notes}

Let $X$ be a complex Banach space and let $B$ be the open unit ball in $X$. Let the Poincar\'e metric $\omega$ on $\Delta$, the open unit disc in the complex plane $C$, be given by
$$\omega(a,b)=\tanh^{-1}\frac{|a-b|}{|1-\bar{a}b|} \ \ \ \ \ |a|, |b|<1.$$

Let $x,y$ be two points of $B$. An analytic chain joining $x$ and $y$ in $B$ consists of $2n$ points $z'_1,z''_1,...,z'_n,z''_n$ in $\Delta$, and of $n$ holomorphic functions $f_k:\Delta\rightarrow B$, such that 
$$f_1(z'_1)=x, \dots , f_k(z''_k)=f_{k+1}(z'_{k+1}) \ \ \ \ for~ k=1, \dots , n-1,\ \ f_n(z''_n)=y.$$
Since $B$ is connected, given $x$ and $y$, an analytic chain joining $x$ and $y$ in $B$ always exists, provided that $n$ is sufficiently large. Let
$$K_B(x,y)=\inf\{\omega(z'_1,z''_1)+\omega(z'_2,z''_2)+\dots +\omega(z'_n,z''_n)\},$$
where the infimum is taken over all choices of analytic chains joining $x$ and $y$ in $B$. It is called the Kobayashi pseudo-metric on $B$. 

The next theorems from \cite{fran} will be used later on.
\vspace{0.5 cm}
\begin{theorem}\label{iso}
Let $B$ be the open unit ball of the complex Banach space $X$. If $F:B\rightarrow B$ is a bi-holomorphic map, then
$$K_B(F(x),F(y))=K_B(x,y),\ \ \ for~all~x,y\in B.$$
\end{theorem}

\begin{theorem}\label{all metric}
Let $B$ be the open unit ball of a complex Banach space $X$. Then
$$K_B(0,x)=\omega(0,\|x\|).$$
\end{theorem}

Let $H$, $K$ be complex Hilbert spaces. We denote the open unit ball of $B(K,H)$, the space of all bounded linear operators from $K$ to $H$, by $\mathcal{B}$.   For each $A \in\mathcal{B}$, we define a transformation $\eta$ on $\mathcal{B}$ setting
\begin{equation}\label{linear transformation}
\eta(Z)=(I-AA^*)^{-\frac{1}{2}}(Z+A)(I+A^*Z)^{-1}(I-A^*A)^{\frac{1}{2}}.
\end{equation}
We collect the facts about this  transformation that we need in the following lemma.
\vspace{0.5 cm}
\begin{lemma}\label{llt} If $\eta$ is as in (\ref{linear transformation}), then $\eta$ has the following properties:

\begin{itemize}
\item [(i)]$\eta$ is invertible and its inverse is given by
$$\eta^{-1}(Z)=(I-AA^*)^{-\frac{1}{2}}(Z-A)(I-A^*Z)^{-1}(I-A^*A)^{\frac{1}{2}}.$$
\item[(ii)] $\eta$ is a bi-holomorphic mapping on $\mathcal{B}$.
\item[(iii)] If $\dim(K)<\infty$, then $\eta$ is WOT-continuous.
\end{itemize}
\end{lemma}
\textbf{Proof.}
To prove statement (i), take any $Z\in\mathcal{B}$ and insert
$$Y=(I-AA^*)^{-\frac{1}{2}}(Z-A)(I-A^*Z)^{-1}(I-A^*A)^{\frac{1}{2}}.$$
We shall show  that $\eta(Y)=Z$. To do this we should note that
\begin{equation}\label{commut}
Y=(I-AA^*)^{\frac{1}{2}}(I-ZA^*)^{-1}(Z-A)(I-A^*A)^{-\frac{1}{2}}.
\end{equation}
For, (\ref{commut}) is equivalent to the identity
$$(Z-A)(I-A^*A)^{-1}(I-A^*Z)=(I-ZA^*)(I-AA^*)^{-1}(Z-A),$$
which can be easily verified by a straightforward calculation. 
Now, we evaluate the operator $I+A^*Y$:
\begin{eqnarray}
I+A^*Y&=&I+A^*(I-AA^*)^{\frac{1}{2}}(I-ZA^*)^{-1}(Z-A)(I-A^*A)^{-\frac{1}{2}}\nonumber\\
~&=&I+(I-A^*A)^{\frac{1}{2}}(I-A^*Z)^{-1}(A^*Z-A^*A)(I-A^*A)^{-\frac{1}{2}}\nonumber\\
~&=&I-I+(I-A^*A)^{\frac{1}{2}}(I-A^*Z)^{-1}(I-A^*A)^{\frac{1}{2}}\nonumber\\
~&=&(I-A^*A)^{\frac{1}{2}}(I-A^*Z)^{-1}(I-A^*A)^{\frac{1}{2}}.\nonumber
\end{eqnarray} 
From this equation we see at once that
\begin{eqnarray}
\eta(Y)&=&(I-AA^*)^{-\frac{1}{2}}(Y+A)(I-A^*A)^{-\frac{1}{2}}(I-A^*Z)\nonumber\\
~&=&[Z(I-A^*Z)^{-1}-ZA^*(I-ZA^*)^{-1}A](I-A^*A)^{-1}(I-A^*Z)\nonumber\\
~&=&Z(I-A^*Z)^{-1}[I-A^*A](I-A^*A)^{-1}(I-A^*Z)\nonumber\\
~&=&Z.\nonumber
\end{eqnarray}

Statement (ii) was shown in theorem 2 of \cite{harr}.
To prove statement (iii): It is easy to check that if $K$ is finite dimensional, then the norm topology of $B(K,H)$ coincides with the strong operator topology while WOT coincides with the weak topology of $B(K,H)$. In this case, the WOT-continuity of $\eta$ was noticed and used by Krein \cite{krei}.
 
~~~~~~~~~~~~~~~~~~~~~~~~~~~~~~~~~~~~~~~~~~~~~~~~~~~~~~~~~~~~~~~~~~~~~~~~~~~~~~~~~~~~~~~~~~~~~~~~~~~~~~~~~~~~~~~~~~~~~~~~~~~~~~~~~~~~~~~~~~\fbox\\
\section{The metric space $C(H,K)$}
We denote the space of all closed densely-defined linear operators from $H$ to $K$ by $C(H,K)$. The first relaxation in the concept of operator is to not assume that the operators are defined everywhere on $H$. Hence,  densely-defined operator $T:H\rightarrow K$ is a linear function whose domain of definition is dense linear subspace $\mathcal{D}(T)$  in $H$. 
$T$ is closed if its graph, $\mathcal{G}(T)$, is a closed subset of space $H\oplus K$ \cite{schm}.

Let $T\in C(H,K)$ and define $L_T$ and $R_T$ settings as: 
$$L_T(X)=(I+T^*T)^{\frac{1}{2}}X^*-T^*(I+XX^*)^\frac{1}{2},$$
$$R_T(X)=(I+XX^*)^\frac{1}{2}(I+TT^*)^\frac{1}{2}-XT^*.$$

Consider operator $X$ such that the compositions are well-defined. Using lemma 1.10 of Schm\"udgen \cite{schm}, recall that $\mathcal{G}(T^*)=V(\mathcal{G}(T))^{\perp}$ where $V(x,y)=(-y,x)$, $x\in H$, $y\in K$. Hence, $K\bigoplus H=\mathcal{G}(T^*)\bigoplus V(\mathcal{G}(T))$. Therefore, for each $u\in H$, there exist $x\in\mathcal{D}(T)$ and $y\in\mathcal{D}(T^*)$ such that $y-Tx=0$, $T^*y+x=u$. That is, $I+T^*T$ is surjective.  $T^*T$ is a positive self-adjoint operator and, for $x\in \mathcal{D}(T^*T)$: 
$$\|(I+T^*T)x\|^2=\|x\|^2+\|T^*Tx\|^2+2\|Tx\|^2,$$
hence, $I+T^*T$ is a bijective mapping with a positive bounded self-adjoint inverse on $H$ such that:
$$0\leq (I+T^*T)^{-1}\leq I.$$  

On the other hand: 
$$(I+T^*T)^{-1}(H)=\mathcal{D}(I+T^*T)=\mathcal{D}(T^*T),$$
hence:
\begin{eqnarray}
\|T(I+T^*T)^{-1}x\|^2 &=& \langle T^*T(I+T^*T)^{-1}x,(I+T^*T)^{-1}x\rangle  \nonumber\\
~&\leq & \langle T^*T(I+T^*T)^{-1}x,(I+T^*T)^{-1}x\rangle \nonumber\\
~&~& \ \ \ \   +  \langle (I+T^*T)^{-1}x,(I+T^*T)^{-1}x\rangle \nonumber\\
~&= &  \langle(I+T^*T)(I+T^*T)^{-1}x,(I+T^*T)^{-1}x\rangle \nonumber\\
~&=&\langle x,(I+T^*T)^{-1}x\rangle =\|(I+T^*T)^{-\frac{1}{2}}x\|^2,  \nonumber
\end{eqnarray}
that is:\\
\centerline{$\|T(I+T^*T)^{-\frac{1}{2}}y\|\leq\|y\|$\ \  \ for $y\in (I+T^*T)^{-\frac{1}{2}}(H)$.}
Since $(I+T^*T)^{-\frac{1}{2}}$ is bijective, $(I+T^*T)^{-\frac{1}{2}}H$ is dense in $H$. Operator $T(I+T^*T)^{-\frac{1}{2}}$ is closed since $T$ is closed and $(I+T^*T)^{-\frac{1}{2}}$ is bounded \cite{schm}. This implies that $\mathcal{D}(T(I+T^*T)^{-\frac{1}{2}})=H$, and $\|T(I+T^*T)^{-\frac{1}{2}}\|\leq1$. A similar argument shows that $\|(I+T^*T)^{-\frac{1}{2}}T^*\|\leq 1$; however, if $K$ is finite dimensional, operator $TT^*$ is bounded and 
$$\|T(I+T^*T)^{-\frac{1}{2}}\|^2=\|TT^*(I+TT^*)^{-1}\|=\frac{\|TT^*\|}{1+\|TT^*\|}<1.$$

It is known that $\|f(TT^*)\|=f(\|TT^*\|)$ if function $f$ is non-decreasing on the interval $[0,\|TT^*\|]$ and $TT^*$ is a positive operator; thus, if $K$ is a finite dimensional Hilbert space, the inverse of the operator $R_S(T)$ exists, and: 
$$R_S(T)^{-1}=(I+SS^*)^{-\frac{1}{2}}[I-T(I+T^*T)^{-\frac{1}{2}}(I+S^*S)^{-\frac{1}{2}}S^*]^{-1}(I+TT^*)^{-\frac{1}{2}}.$$ 

\begin{remmark}
From the above it follows that if $K$ is finite dimensional and $T\in C(H,K)$, then $\hat{T}=(I+T^*T)^{-\frac{1}{2}}T^*\in\mathcal{B}$. If $A\in\mathcal{B}$, then $\ker(I-A^*A)={0}$. Because, if $A^*A(x)=x$ then 
$$\langle A^*A(x),x\rangle=\langle x,x\rangle,$$
which meanse $\|A(x)\|=\|x\|$ and $x$ must be zero. So, $A_0=(I-A^*A)^{-\frac{1}{2}}A^*$ is a closed densely-defined operator from $H$ to $K$, such that $\hat{A_0}=A$.
\end{remmark}
\vspace{0.5 cm}
\begin{lemma}\label{mobius}
If  $T\in C(H,K)$, then $\psi_T$, by the following definition, is of the form of (\ref{linear transformation}), $\psi_T(\hat{T})=0$, $\psi_T(0)=-\hat{T}$ and $\psi^{-1}_T=\psi_{-T}$.\\ For each $X\in \mathcal{B}$,  
$\psi_T (X)=L_T(Y)R_Y(T)^{-1}$ where $Y=(1-X^*X)^{-\frac{1}{2}}X^*$.
\end{lemma}
\textbf{Proof.}
It is easy to see that $\psi_T$ is of the form of (\ref{linear transformation}), because:

\leftline{$\psi_T(X)=L_T(Y)R_Y(T)^{-1}$}
\leftline{$=\{(I+T^*T)^\frac{1}{2}Y^*-T^*(I+YY^*)^\frac{1}{2}\}\{(I+TT^*)^\frac{1}{2}(I+YY^*)^\frac{1}{2}-TY^*\}^{-1}$ }
\leftline{$=\{(I+T^*T)^\frac{1}{2}Y^*-T^*(I+YY^*)^\frac{1}{2}\}(I+YY^*)^{-\frac{1}{2}}\{(I+TT^*)^\frac{1}{2}-T\hat{Y}\}^{-1}$ }
\leftline{$=\{(I+T^*T)^\frac{1}{2}\hat{Y}-T^*\}\{(I+TT^*)^\frac{1}{2}-T\hat{Y}\}^{-1}$ }
\leftline{$=\{[(I+T^*T)^{-1}(I+T^*T-T^*T)]^{-\frac{1}{2}}\hat{Y}-T^*\}\{(I+TT^*)^\frac{1}{2}-T\hat{Y}\}^{-1}$ }
\leftline{$=\{[I-\hat{T}\hat{T}^*]^{-\frac{1}{2}}\hat{Y}-T^*\}\{(I+TT^*)^\frac{1}{2}-T\hat{Y}\}^{-1}$ }
\leftline{$=\{(I-\hat{T}\hat{T}^*)^{-\frac{1}{2}}\hat{Y}-[I+T^*T-T^*T]^{\frac{1}{2}}T^*\}\{(I+TT^*)^\frac{1}{2}-T\hat{Y}\}^{-1}$ }
\leftline{$=\{(I-\hat{T}\hat{T}^*)^{-\frac{1}{2}}\hat{Y}-(I-\hat{T}\hat{T}^*)^{-\frac{1}{2}}\hat{T}\}\{(I+TT^*)^\frac{1}{2}-T\hat{Y}\}^{-1}$ }
\leftline{$=(I-\hat{T}\hat{T}^*)^{-\frac{1}{2}}(\hat{Y}-\hat{T})\{(I+TT^*)^\frac{1}{2}-T\hat{Y}\}^{-1}$ }
\leftline{$=(I-\hat{T}\hat{T}^*)^{-\frac{1}{2}}(\hat{Y}-\hat{T})\{I-(I+TT^*)^{-\frac{1}{2}}T\hat{Y}\}^{-1}(I+TT^*)^{-\frac{1}{2}}$ }
\leftline{$=(I-\hat{T}\hat{T}^*)^{-\frac{1}{2}}(\hat{Y}-\hat{T})\{I-\hat{T}^*\hat{Y}\}^{-1}(I+TT^*)^{-\frac{1}{2}}$ }
\leftline{$=(I-\hat{T}\hat{T}^*)^{-\frac{1}{2}}(\hat{Y}-\hat{T})(I-\hat{T}^*\hat{Y})^{-1}\{(I+TT^*)^{-1}[I+TT^*-TT^*]\}^\frac{1}{2}$ }
\leftline{$=(I-\hat{T}\hat{T}^*)^{-\frac{1}{2}}(\hat{Y}-\hat{T})(I-\hat{T}^*\hat{Y})^{-1}\{I-\hat{T}^*\hat{T}\}^\frac{1}{2}$ }
\leftline{$=(1-\hat{T}\hat{T}^*)^{-\frac{1}{2}}(X-\hat{T})(1-\hat{T}^*X)^{-1}(1-\hat{T}^*\hat{T})^{\frac{1}{2}},$ }

\leftline{where $\hat{T}=(1+T^*T)^{-\frac{1}{2}}T^*\in \mathcal{B}$.}
\leftline{Therefore, by lemma \ref{llt}, $\psi^{-1}_T=\psi_{-T}$, $\psi_T(\hat{T})=0$ and $\psi_T(0)=-\hat{T}$.~~~\fbox} 

\begin{theorem}
If $H$, $K$ are complex Hilbert spaces and $\dim K<\infty$,  then
$$d(T,S)=\tanh ^{-1}\|L_T(S)R_{S}(T)^{-1}\|.$$
defines a metric on $C(H,K)$. This metric satisfies the next equality.
$$d(T,S)=K_{\mathcal{B}}(\hat{T},\hat{S})$$
\end{theorem}
\begin{proof}
It is easy to see that, in this case, by theorems \ref{iso}, \ref{all metric} and  lemma \ref{mobius}, $d$ defines a metric on $C(H,K)$. And for each $T,S\in C(H,K)$, we have: 
\begin{eqnarray}
d(T,S)&=&\tanh^{-1}\|L_T(S)R_S(T)^{-1}\|=\omega(0,\|\psi_T(\hat{S})\|)\nonumber\\
~&=&K_{\mathcal{B}}(0,\psi_T(\hat{S}))=K_{\mathcal{B}}(\psi_T(\hat{T}),\psi_T(\hat{S}))=K_{\mathcal{B}}(\hat{T},\hat{S}),\nonumber
\end{eqnarray}
where $\psi_T$ is as in lemma \ref{mobius}.
\end{proof}

\section{Complex symmetric operators from $H$ to $K$}
Suppose that $H$ and $K$ are two complex separable 
Hilbert space endowed with a conjugation pair $(C_1,C_2)$ from $H$ to $K$. Specifically, this means that $C_1:H\rightarrow K$ and $C_2:K\rightarrow H$ are conjugate linear operators that is $C_2C_1=id_H$ or $C_1C_2=id_K$, and they are involutive with each other, meaning that $\langle C_1x,y\rangle_K=\langle C_2y,x\rangle_H$ holds for all $x\in H$ and $y\in K$. 

An operator $T:H\rightarrow K$ is called $(C_1,C_2)$-symmetric if $C_2T=T^*C_1$ whenever $C_2C_1=id_H$, and $TC_2=C_1T^*$ whenever $C_1C_2=id_K$. More generally, $T$ is called complex symmetric if it is $(C_1,C_2)$-symmetric  with respect to some conjugation pair $(C_1,C_2)$. 

In the following, we let $CSO_u(H,K)$ and $CSO_b(H,K)$ denote the set of all complex symmetric operators of $C(H,K)$ and $B(H,K)$, respectively.

The terminology arises from  the fact that if $H=K$ and $C_1=C_2=C$, $T$ is $(C,C)$-symmetric if and only if it is $C$-symmetric. In particular, an $n\times m$ ($n\geq m$) matrix $T$ contains a symmetric $m\times m$ block if and if $C_2T=T^*C_1$ where $C_1$ and $C_2$ are define by:
\begin{eqnarray}
C_1(z_1,\dots,z_m)&=&(\overline{z_1},\overline{z_2},\dots ,\overline{z_m},0,\dots , 0),\nonumber\\
C_2(z_1,\dots,z_n)&=&(\overline{z_1},\overline{z_2},\dots ,\overline{z_m}),\nonumber
\end{eqnarray}
for $z_i\in\mathbb{C}$.

\begin{lemma}
Every closed densely-defined operator from a Hilbert space $H$ to a Hilbert space $K$, has a complex symmetric extension.
\end{lemma}
\begin{proof}
Let $T\in C(H,K)$ and $(C_1,C_2)$ be an arbitrary conjugation pair from $H$ to $K$. The operator $\tilde{T}$ from $H\oplus H$ to $K\oplus K$ is a $(\tilde{C_1},\tilde{C_2})$-symmetric extension of $T$, where $\tilde{T}$, $\tilde{C_1}$ and $\tilde{C_2}$ are defined by:
\begin{equation}\label{tild}
\tilde{T}=\begin{pmatrix}
T & 0 \\
0 & C_1T^*C_1
\end{pmatrix}, 
~~~\tilde{C_1}=\begin{pmatrix}
0 & C_1 \\
C_1 & 0
\end{pmatrix}, 
~~~\tilde{C_2}\begin{pmatrix}
0 & C_2 \\
C_2 & 0
\end{pmatrix}.
\end{equation}
 This completes the proof.
\end{proof}
\begin{lemma}\label{a}
Let $A$ be a bounded $(C_1,C_2)$-symmetric operator from $K$ to $H$, with $\|A\|<1$. Then $T=(I-A^*A)^{-\frac{1}{2}}A^*$ is a closed densely-defined complex symmetric operator from $H$ to $K$. 
\end{lemma}
\begin{proof}
For $x\in\mathcal{D}(T)$, we compute
\begin{eqnarray}\label{c}
\|Tx\|^2+\|x\|^2&=&\langle AA^*(I-AA^*)^{-\frac{1}{2}}x,(I-AA^*)^{-\frac{1}{2}}x\rangle+\langle x,x\rangle\nonumber\\
~&=&-\langle (I-AA^*)(I-AA^*)^{-\frac{1}{2}}x,(I-AA^*)^{-\frac{1}{2}}x\rangle +\langle x,x\rangle\nonumber\\
~&~&+\langle (I-AA^*)^{-\frac{1}{2}}x,(I-AA^*)^{-\frac{1}{2}}x\rangle\nonumber\\
~&=&\|(I-AA^*)^{-\frac{1}{2}}x\|^2.
\end{eqnarray}
We prove that $T$ is closed. Suppose that $x_n\rightarrow x$ and $Tx_n\rightarrow y$ for a sequence of vectors $x_n\in\mathcal{D}(T)$. Then, by (\ref{c}), $\{(I-AA^*)^{-\frac{1}{2}}x_n\}$ is a Cauchy sequence, so it converges in $H$. Since $(I-AA^*)^{-\frac{1}{2}}$ is self-adjoint and hence closed, we get $x\in\mathcal{D}((I-AA^*)^{-\frac{1}{2}})=\mathcal{D}(T)$ and $(I-AA^*)^{-\frac{1}{2}}x_n\rightarrow (I-AA^*)^{-\frac{1}{2}}x$. Hence, 
$$Tx_n=A^*(I-AA^*)^{-\frac{1}{2}}x_n\rightarrow A^*(I-AA^*)^{-\frac{1}{2}}x=Tx=y.$$
This proves that the operator $T$ is closed. Clearly, $T$ is densely-defined, because the self-adjoint operator $(I-AA^*)^{-\frac{1}{2}}$ is densely-defined. Therefore, $T\in C(H,K)$. Now, we are ready to prove that $T$ is symmetric.

Let $C_2C_1=id_K$, and define $\mathcal{C}_1$ and $\mathcal{C}_2$ settings as:
\begin{eqnarray}
\mathcal{C}_1&=&(I-A^*C_1C_2A)^{-\frac{1}{2}}C_2(I-AA^*)^\frac{1}{2},\nonumber\\
\mathcal{C}_2&=&(I-AA^*)^\frac{1}{2}C_1(I-A^*C_1C_2A)^{-\frac{1}{2}}.\nonumber
\end{eqnarray}
For each $x\in K$ and $y\in H$, $\|C_1x\|=\|x\|$ and $\|C_2y\|\leq \|y\|$. So, $I-A^*C_1C_2A$ is positive self-adjoint and $\ker(I-A^*C_1C_2A)=\{0\}$. Because, $\|A\|<1$ and if $A^*C_1C_2Ax=x$ then
\begin{eqnarray}
\langle x,x\rangle &=& \langle A^*C_1C_2Ax,x\rangle\nonumber\\
~&=&\langle C_1C_2Ax,Ax\rangle\nonumber\\
~&=&\langle C_2Ax,C_2Ax\rangle\nonumber\\
~&=&\| C_2Ax \|^2\leq \| Ax\|^2,\nonumber
\end{eqnarray}
which means $\|Ax\|=\|x\|$ and $x$ must be zero. This proves that $\mathcal{C}_1$ and $\mathcal{C}_2$ are well-define. It is obviouse that $(\mathcal{C}_1,\mathcal{C}_2)$ is a conjugation pair from $H$ to $K$ and $T$ is a $(\mathcal{C}_1,\mathcal{C}_2)$-symmetric operator. Because:
\begin{eqnarray}
\mathcal{C}_1\mathcal{C}_2&=&(I-A^*C_1C_2A)^{-\frac{1}{2}}C_2(I-AA^*)C_1(I-A^*C_1C_2A)^{-\frac{1}{2}}\nonumber\\
~&=&(I-A^*C_1C_2A)^{-\frac{1}{2}}(C_2C_1-A^*C_1C_2A)(I-A^*C_1C_2A)^{-\frac{1}{2}}=id_K,\nonumber
\end{eqnarray} 
and
\begin{eqnarray}
T\mathcal{C}_2&=&(I-A^*A)^{-\frac{1}{2}}A^*(I-AA^*)^\frac{1}{2}C_1(I-A^*C_1C_2A)^{-\frac{1}{2}}\nonumber\\
~&=&A^*C_1(I-A^*C_1C_2A)^{-\frac{1}{2}}\nonumber\\
~&=&C_2A(I-A^*C_1C_2A)^{-\frac{1}{2}}\nonumber\\
~&=&(I-C_2AA^*C_1)^{-\frac{1}{2}}C_2A\nonumber\\
~&=&(I-A^*C_1C_2A)^{-\frac{1}{2}}C_2(I-AA^*)^\frac{1}{2}(I-AA^*)^{-\frac{1}{2}}A\nonumber\\
~&=&\mathcal{C}_1T^*.\nonumber
\end{eqnarray}
\end{proof}

 We are now ready to prove the main result of this section. Our approach is inspired by the arguments of \cite{ez}. 
\begin{theorem}
If $H$, $K$ are complex separable Hilbert spaces, $H$ is infinite-dimensional and $\dim K<\infty$, then 
$$\overline{CSO}^{d}_u=C(H,K).$$
\end{theorem}
\begin{proof}
Let $T\in C(H,K)$, fix an orthonormal basis $\alpha=\{f_1,\dots,f_t\}$ of $K$ and an orthonormal basis $\beta=\{e_1,e_2,\dots\}$ of $H$, and let $H_n=\langle e_1,\dots e_n\rangle$. Define $A_n\in B(K,H_n)$ by insisting that $\langle A_nf_k,e_j\rangle=\langle\hat{T}f_k,e_j\rangle$, for $1\leq k\leq t$ and $1\leq j\leq n$. In other words, $A_n$ is simply the upper-left $n\times k$ principal submatrix of the matrix representation of $\hat{T}$ with respect to $\alpha$ and $\beta$. Let $(C_n^1,C_n^2)$ be an arbitrary conjugation pair from $K$ to $H$ and observe that the operator $\tilde{A_n}$ from $K\oplus K$ to $H\oplus H$ is complex symmetric by \ref{tild}. Since $n>i$ implies that
$$\|\hat{T}f_i-\tilde{A_n}f_i\|^2=\sum_{j=n+1}^\infty |\langle\hat{T}f_i,e_j\rangle|^2,$$
it follows that $\tilde{A_n}f_i\rightarrow \hat{T}f_i$ for each fixed $i$. Since $\|\tilde{A_n}\|\leq \|\hat{T}\|$ by construction, it follows from proposition IX. 1.3. (d) nof \cite{conway} that $\tilde{A_n}\rightarrow\hat{T}$ (SOT). Since the weak-operator topology is weaker than the strong-operator topology, we also have the convergence in the weak-operator topology. Therefore, by lemma\ref{llt} and lemma\ref{mobius}, $\psi_T(\tilde{A_n})\rightarrow 0$ (WOT), and hence $\lim_n\|\psi_T(\tilde{A_n})\|=0$. Because $K$ is finite dimensional and 
$$\|\psi_T(\tilde{A_n})\|\leq \sum_{i=1}^t\|\psi_T(\tilde{A_n})f_i\|.$$
Insert $T_n=(I-\tilde{A_n}^*\tilde{A_n})^{-\frac{1}{2}}\tilde{A_n}^*$. By lemma \ref{mobius}, we have 
\begin{eqnarray}
d(T_n,T)&=&\tanh^{-1}\|L_{T_n}(T)R_{T}(T_n)^{-1}\|\nonumber\\
~&=&\tanh^{-1}\|\psi_T(\hat{T_n})\| =\tanh^{-1}\|\psi_T(\tilde{A_n})\|.\nonumber
\end{eqnarray}
 This implies that $\lim_nd(T_n,T)=0$. By lemma \ref{a}, $T_n$ is complex symmetric operator, and this completes the proof.
\end{proof}
Now, we consider the weak operator topology (WOT), strong operator topology (SOT), and strong-* topology (SST), on $B(H,K)$. You should note that by using the proof of previous theorem, $\overline{CSO}^{SST}_b =B(H,K)$. Since the strong and weak operator topologies are both weaker than the strong-* topology, we have the next corollary. 

\begin{corollary}
If $H$, $K$ are complex separable, infinite-dimensional Hilbert spaces, then
$$\overline{CSO}^{(SST)}_b=\overline{CSO}^{(SOT)}_b=\overline{CSO}^{WOT}_b=B(H,K)$$
\end{corollary}
 
\end{document}